\documentclass[a4paper]{article}
\usepackage{amsmath}
\usepackage{amssymb}
\usepackage{amsopn}
\usepackage{amstext}
\usepackage{amsthm}
\usepackage{amscd}

\newtheorem{lemma}{Lemma}[section]
\newtheorem{theorem}[lemma]{Theorem}

\newtheorem{remark}[lemma]{Remark}
\newtheorem{corollary}[lemma]{Corollary}
\newtheorem{proposition}[lemma]{Proposition}

\def\Ext{\mathrm{Ext}}
\def\R{\mathbb R}
\def\I{\mathcal I}
\def\D{\mathcal D}
\def\N{\mathcal N}
\def\Hom{\mathrm{Hom}}
\def\M{\mathrm M}
\def\E{\mathcal E}
\def\Ker{\mathrm{Ker}\,}
\def\Ran{\mathrm{Ran}\,}
\def\diag{\mathrm{\,diag\,}}
\def\A{\mathcal A}
\def\B{\mathcal B}
\def\C{\mathcal C}
\def\J{\mathcal J}
\def\Z{\mathbb Z}
\def\K{\mathcal K}

\begin{document}

\title{{\bf The Ext--Group of Unitary Equivalence Classes of Unital Extensions}}
\author{{\bf Yifeng Xue}\\
{\it \small Department of Mathematics, East China Normal University, Shanghai 200241, P.R.China}\\
{\it\small email: yfxue@math.ecnu.edu.cn}}
\date{}

\maketitle
\renewcommand{\abstractname}{\ }
\vspace*{-4mm}
\begin{abstract}
\noindent{\bf Abstract} Let $\A$ be a unital separable nuclear $C^*$--algebra which belongs to the bootstrap category
$\N$ and $\B$ be a separable stable $C^*$--algebra. In this paper, we consider the group
$\Ext_u(\A,\B)$ consisting of the unitary equivalence classes of unital extensions
$\tau\colon\A\rightarrow Q(\B)$. The relation between $\Ext_u(\A,\B)$ and
$\Ext(\A,\B)$ is established. Using this relation, we show the half--exactness of
$\Ext_u(\cdot,\B)$ and the (UCT) for $\Ext_u(\A,\B)$. Furthermore, under certain conditions,
we obtain the half--exactness and Bott periodicity of $\Ext_u(\A,\cdot)$.
\vspace*{1.5mm}

\vspace*{1mm}

\noindent{\bf MR(2000) Subject Classification} 46L05, 46L80, 46L35
\end{abstract}

\footnote[0]{\hspace*{-5mm}Received May 12, 2009, Accepted July 27, 2010}
\footnote[0]{\hspace*{-5mm}supported by Natural Science Foundation of China (Grant no. 10771069) and
Shanghai Leading Academic Discipline Project(Grant no. B407)}
\baselineskip 16pt

\section{Introduction and Preliminaries}

For a $C^*$--algebra $\E$, let $\M_n(\E)$ (resp. $\E^{1\times n}$) denote the set of all
$n\times n$ (resp. $1\times n$) matrices over $\E$. If $a=(a_1,\cdots,a_n)\in\E^{1\times n}$,
we set $a^T=\begin{pmatrix}a_1^*\\ \vdots\\ a_n^*\end{pmatrix}$. Suppose $\E$ has unit $1$.
Denote by $U(\E)$ (resp. $U_0(\E)$) the unitary group of $\E$ (resp. the connected component
of $1$ in $U(\E)$). The definitions of $K$--groups of $\E$ can be found in \cite{Bl}.
Throughout the paper, $\A$ is a separable unclear $C^*$--algebra with unit $1_\A$ and $\B$
is a separable stable $C^*$--algebra.

Let $M(\B)$ be the multiplier of $\B$ and set $Q(\B)=M(\B)/\B$. Let $\pi\colon M(\B)\rightarrow
Q(\B)$ be the quotient map. It is well--known that the $C^*$--algebra extension of $\A$ by $\B$
can be identified with $\Hom(A,Q(\B))$ (the set of all $*$--homomorphisms from $\A$ to $Q(\B)$).
$\tau\in\Hom(A,Q(\B))$ is called to be unital extension if $\tau(1_\B)=1_{Q(\B)}$. Let
$\Hom_1(A,Q(\B))$ be the set of all unital extensions. $\tau\in\Hom(A,Q(\B))$ (resp.
$\Hom_1(A,Q(\B))$) is trivial, if there is a $*$--(resp. unital) homomorphism $\phi\colon\A
\rightarrow M(\B)$ such that $\tau=\pi\circ\phi$.

Two extensions $\tau_1,\tau_2\in\Hom(A,Q(\B))$ are unitarily equivalent (denoted by $\tau_1\sim_u
\tau_2$) if there is a unitary $u\in M(\B)$ such that $Ad_{\pi(u)}\circ\tau_=\tau_2$. Let
$[\tau]$ (resp. $[\tau]_u$) denote the unitary equivalence of $\tau$ in $\Hom(A,Q(\B))$ (resp.
$\Hom_1(A,Q(\B))$) and set $\Ext(\A,\B)=\{[\tau]\vert\,\tau\in\Hom(A,Q(\B))\}$, $\Ext_u(\A,\B)=
\{\tau]_u\vert\,\tau\in\Hom_1(A,Q(\B))\}$).

Since $\B$ is table, there are isometries $u_1,\cdots,u_n\in M(\B)$ such that $\sum\limits^n_{i=1}
u_iu_i^*=1_{M(\B)}$ (here $n=1,2,\cdots$). Let $\tau_1,\tau_2\in\Hom(A,Q(\B))$
(or $\Hom_1(A,Q(\B))$). Define $\tau_1\oplus\tau_2\in\Hom(A,Q(\B))$ (or $\Hom_1(A,Q(\B))$) by
$$
(\tau_1\oplus\tau_2)(a)=(\pi(u_1),\pi(u_2))\diag(\tau_1(a),\tau_2(a))(\pi(u_1),\pi(u_2))^T,\
\forall\,a\in\A.
$$
$[\tau_1\oplus\tau_2]$ (or $[\tau_1\oplus\tau_2]_2$) is independent of the choice of $u_1,u_2$
and equivalence classes of $[\tau_i]$ (or $[\tau_i]_u$), $i=1,2$. So we can define an addition
in $\Ext(\A,\B)$ (resp. $\Ext_u(\A,\B)$) by $[\tau_1]+[\tau_2]=[\tau_1\oplus\tau_2]$ (resp.
$[\tau_1]_u+[\tau_2]=[\tau_1\oplus\tau_2]_u$).

Let $\tau_1,\tau_2\in\Hom(A,Q(\B))$ (resp. $\Hom_1(A,Q(\B))$). $[\tau_1]=[\tau_2]$ (resp.
$[\tau_1]_u=[\tau_2]_u$) means that there are (resp. unital) trivial extensions $\tau_0,\tau_0'$
such that $\tau_1\oplus\tau_0\sim_u\tau_2\oplus\tau_0'$. In this case, $\Ext(\A,\B)$ and
$\Ext_u(\A,\B)$ become Abelian groups by \cite[Corollary 15.8.4]{Bl}.

$\Ext_u(\A,\B)$ and $\Ext(\A,\B)$ are different in general. $\Ext(\A,\B)$ has Bott periodicity
and six--term exact sequences for variables $\A$ or $\B$ and the universal coefficient formula
for $\A$ and $\B$ etc. $\Ext_u(\A,\B)$ has no these properties in general. But there are some
relations between them. For examples, L.G. Brown and M. Dadarlat showed the following
sequences
\begin{align*}
0\longrightarrow\Z/\{h([1_\A])\vert\,h\in\Hom(K_0(\A),\Z)\}&\longrightarrow\Ext_u(\A,\K)
\longrightarrow \Ext(\A,\K)\longrightarrow 0\\
0\longrightarrow\Ext(K_0(\A),[1_\A],\Z)\longrightarrow\Ext_u&(\A,\K)\longrightarrow
\Hom(K_0(\A),\K)\longrightarrow 0
\end{align*}
when $\A$ is in the bootstrap category $\N$ (cf. \cite[Proposition 1, Theorem 2]{BD});
V. Manuilov and K. Thomsen in \cite{MT} presented the six--term exact sequence for
$\Ext(\A,\B)$ and $\Ext_u(\A,\B)$ as follows
$$
\begin{CD}
K_0(\B)@> >>\Ext_u(\A,\B)@> >>\Ext(\A,\B)\\
@A AA @. @V VV\\
\Ext(\A,S\B)@<< <\Ext_u(\A,S\B)@<< <K_1(\B)
\end{CD}
$$
and H. Lin characterized the strongly unitary equivalence of two full essential extensions in
\cite{Lin1} for $\A\in\N$ by means of the subgroup $H_1(K_0(\A),K_0(\B))$ of $K_0(\B)$ (see
below).

In spirited by above results, we will use the subgroups
$$
H_1(K_0(\A),K_i(\B))=\{h([1_\A])\vert\,h\in\Hom(K_0(\A), K_i(\B))\}
$$
of $K_i(\B)$, $i=0,1$, to give an exact sequence for $\Ext_u(\A,\B)$ and $\Ext(\A,\B)$ as follows
\begin{align*}
0\longrightarrow H_1(K_0(\A),K_0(\B))\longrightarrow &K_0(\B)\longrightarrow\Ext_u(\A,\B)\longrightarrow\\
\longrightarrow&\Ext(\A,\B)\longrightarrow H_1(K_0(\A),K_1(\B))\longrightarrow 0
\end{align*}
in this paper when $\A\in\N$. Using this exact sequence, we present the half--exactness of
$\Ext_u(\cdot,\B)$ and the (UCT) for $\Ext_u(\A,\B)$. Furthermore, under certain conditions,
we obtain the half--exactness and Bott periodicity of $\Ext_u(\A,\cdot)$.

\section{The main result}

Let $p, q$ be projections in the $C^*$--algebra $\E$. $p$ and $q$ are equivalent in $\E$, denoted
by $p\sim q$ if there is $u\in\E$ such that $p=u^*u$ and $q=uu^*$. Since $\B$ is separable and
stable, there are isometries $u_1,\cdots,u_n$ in $M(\B)$ such that $\sum\limits^n_{i=1}u_iu_i^*
=1_{M(\B)}=1$, $\forall\,n\ge 2$. Thus we have $K_i(M(\B))\cong 0$, $i=0,1$ by \cite{Bl} and
$\diag(p,1)\sim\diag(1,1)=1_2$ for any projection $p\in M(\B)$. Moreover, the index map
$\partial_0\colon K_1(Q(\B))\rightarrow K_0(\B)$ and the exponential map $\partial_1\colon
K_0(Q(\B))\rightarrow K_1(\B)$ given in \cite{Bl} are isomorphic.

By \cite{Lin2} or \cite{Th}, there is a trivial extension $\tau_{\A,\B}\in\Hom_1(\A,Q(\B))$
such that $\tau_{\A,\B}\oplus\tau_0\sim_u\tau_{\A,\B}$ for any unital trivial extension $\tau_0$,
i.e., $\tau_{\A,\B}$ is a unital absorbing trivial extension. Thus, $[\tau_1]_u=[\tau_2]_u$
in $\Ext_u(\A,\B)$ iff $\tau_1\oplus\tau_{\A,\B}\sim_u\tau_2\oplus\tau_{\A,\B}$.

\begin{lemma}\label{L2a}
Let $\E$ be a $C^*$--algebra with unit $1$. Assume that there are isometries $v_1,\cdots,v_n
\in\E$ with $\sum\limits^n_{i=1}v_iv_i^*=1$. Let $p$ a projection in $\M_n(\E)$ and $u$ be a
unitary in $\M_n(\E)$. Set
$$
q=(v_1,\cdots,v_n)p(v_1,\cdots,v_n)^T,\ w=(v_1,\cdots,v_n)u(v_1,\cdots,v_n)^T.
$$
Then $q$ is a projection in $\E$ and $w$ is unitary in $\E$. Furthermore, $[q]=[p]$ in
$K_0(\E)$ and $[w]=[u]$ in $K_1(\E)$.
\end{lemma}
\begin{proof} It is easy to check that $q$ is a projection and $w$ is unitary in $\E$.
Set $X\!=\!\begin{pmatrix}v_1&\cdots&v_n\\ 0&\cdots&\\ \hdotsfor{3}\\ 0&\cdots&0\end{pmatrix}\! p$.
Then $X^*X=p$ and $XX^*=\diag(q,0)\in\M_n(\E)$.

The rest comes from \cite[Corollary]{X1} or the proof of Lemma 3.1 in \cite{X2}.
\end{proof}

\begin{lemma}\label{L2b}
Let $\tau\in\Hom(\A,Q(\B))$ be a nonunital extension. If $[\tau(1_\A)]=0$ in $K_0(Q(\B))$,
then there is a unital extension $\tau_0$ such that $\tau\oplus\tau_{\A,\B}\oplus 0\sim_u
\tau_0\oplus 0$. when $\tau$ is trivial, $\tau_0$ can be chosen as a trivial extension.
\end{lemma}
\begin{proof}Put $q=\tau(1_\A)$. Let $u_1,u_2$ be isometries in $M(\B)$ with
$u_1u_1^*+u_2u_2^*=1$. Since $[q]=0$ in $K_0(Q(\B))$, we can find $v\in\M_2(Q(\B))$ such that
$vv^*=1_2$, $v^*v=\diag(q,1)$. Set $V_0=(\pi(u_1),\pi(u_2))v(\pi(u_1),\pi(u_2))^T$ and
$\tau_0=Ad_{V_0}\circ(\tau\oplus\tau_{\A,\B})$. Then $V_0V_0^*=1$ and $\tau_0(1_\A)=1$.
Note that for any $a\in\A$,
$$
\begin{pmatrix}V_0&0\\ 1-V_0^*V_0&V_0^*\end{pmatrix}\diag((\tau\oplus\tau_{\A,\B})(a),0)
\begin{pmatrix}V_0^*&1-V_0^*V_0\\ 0&V_0\end{pmatrix}=\diag(\tau_0(a),0).
$$
The assertion follows.

When $\tau$ is trivial, there is a nonunital $*$--homomorphism $\phi\colon\A\rightarrow M(\B)$
such that $\tau=\pi\circ\phi$. Let $p=\phi(1_\A)$. Since $[p]=0$, there is
$u\in M(\B)$ such that $uu^*=1_2$ and $u^*u=\diag(p,1)$. Let $\psi\colon\A\rightarrow M(\B)$
be a unital $*$--homomorphism such that $\tau_{\A,\B}=\pi\circ\psi$. Set
$U_0=(u_1,u_2)u(u_1,u_2)^T$ and
$$
\psi_0(a)=U_0(u_1,u_2)\diag(\phi(a),\psi(a))(u_1,u_2)^TU_0^*,\ \forall\,a\in\A.
$$
Then $\tau_0=\pi\circ\psi_0$ is a unital trivial extension and $\tau\oplus\tau_{\A,\B}\oplus 0
\sim_u\tau_0\oplus 0$.
\end{proof}

Define the map $\Phi_{\A,\B}\colon K_0(\B)\rightarrow\Ext_u(\A,\B)$ by $\Phi_{\A,\B}(x)=
[Ad_v\circ\tau_{\A,\B}]_u$, where $v\in U(Q(\B))$ such that $\partial_0([v])=x$.

\begin{lemma}\label{L2c}
Keeping symbols as above, we have $\Phi_{\A,\B}$ is a homomorphism and moreover,
\begin{enumerate}
\item[$(1)$] $\Ker\Phi_{\A,\B}\subset H_1(K_0(\A),K_0(\B));$
\item[$(2)$] $\Phi_{\A,\B}=[Ad_v\circ\tau_0]_u$ for any unital trivial extension $\tau_0$.
\end{enumerate}
\end{lemma}
\begin{proof}We first prove that $\Phi_{\A,\B}$ is well--defined. If there is $v'\in U(Q(\B))$
such that $\partial_0([v])=\partial_0([v'])=x$, then $v_0=v'v^*\in U_0(Q(\B))$ and hence
there exists $u_0\in U_0(M(\B))$ such that $\pi(u_0)=v_0$. Thus, $Ad_v\circ\tau_{\A,\B}\sim_u
Ad_{v'}\circ\tau_{\A,\B}$.

Now let $x_1,x_2\in K_0(\B)$ and choose $v_1,v_2\in U(Q(\B))$ such that $\partial_0([v_i])=x_i$,
$i=1,2$. Then $\partial_0([\diag(v_1,v_2)])=x_1+x_2$. Let $u_1,u_2$ be isometries in $M(\B)$
such that $u_1u_1^*+u_2u_2^*=1$. Set $z=(\pi(u_1),\pi(u_2))\diag(v_1,v_2)\pi(u_1),\pi(u_2))^T$.
Then $[z]=[\diag(v_1,v_2)]$ in $K_1(\B)$ by Lemma \ref{L2a} and $\partial_0([z])=x_1+x_2$.
Moreover, we have
$$
[Ad_z\circ(\tau_{\A,\B}\oplus\tau_{\A,\B})]_u=[(Ad_{v_1}\circ\tau_{\A,\B})\oplus(Ad_{v_2}\circ
\tau_{\A,\B})]_u=[Ad_{v_1}\circ\tau_{\A,\B}]_u+[Ad_{v_2}\circ\tau_{\A,\B}]_u
$$
in $\Ext_u(\A,\B)$. Let $w_0\in U(M(\B))$ such that
$\tau_{\A,\B}\oplus\tau_{\A,\B}=Ad_{\pi(w_0)} \circ\tau_{\A,\B}$ and
put $z_0$ $=\pi(w_0^*)z\pi(w_0)$. Then $[z_0]=[z]$ in $K_1(Q(\B))$ and
\begin{align*}
\Phi_{\A,\B}(x_1+x_2)=&[Ad_{z_0}\circ\tau_{\A,\B}]_u=[Ad_z\circ(\tau_{\A,\B}\oplus\tau_{\A,\B}]_u\\
=&[Ad_{v_1}\circ\tau_{\A,\B}]_u+[Ad_{v_2}\circ\tau_{\A,\B}]_u\\
=&\Phi_{\A,\B}(x_1)+\Phi_{\A,\B}(x_2).
\end{align*}
When $x=0$ in $K_0(\B)$, there is $u\in U(M(\B))$ such that $v=\pi(u)$ and consequently,
$[Ad_v\circ\tau_{\A,\B}]_u=0$ in $\Ext_u(\A,\B)$.

Let $e\in\Ker\Phi_{\A,\B}$. Then there is $v\in U(Q(\B))$ such that $\partial_0([v])=e$
and $[Ad_v\circ\tau_{\A,\B}]_u=0$ in $\Ext_u(\A,\B)$. Thus, we can find $u\in U(M(\B))$
such that $(Ad_v\circ\tau_{\A,\B})\oplus\tau_{\A,\B}=Ad_{\pi(u)}\circ\tau_{\A,\B}$. Set
$$
\tilde{v}=(\pi(u_1),\pi(u_2))\diag(v,1)(\pi(u_1),\pi(u_2))^T.
$$
Then $[\tilde{v}]=[v]$ in $K_1(\B)$ by Lemma \ref{L2a} and
$$
Ad_{\pi(u)}\circ\tau_{\A,\B}=Ad_{\tilde{v}}\circ(\tau_{\A,\B}\oplus\tau_{\A,\B})=
Ad_{\tilde{v}}\circ Ad_{\pi(w_0)}\circ\tau_{\A,\B}.
$$
Set $\hat v=\pi(u^*)\tilde{v}\pi(w_0)$. Then $[\hat v]=[v]$ and $Ad_{\hat v}\circ\tau_{\A,\B}
=\tau_{\A,\B}$. Therefore, we can define a unital extension $\hat\tau\colon C(\mathbf S^1,\A)
\rightarrow Q(\B)$ by $\hat\tau(a)=\tau_{\A,\B}(a)$ for $a\in\A$ and $\hat\tau(z1_\A)=\hat v$,
$z\in\mathbf S^1$.

Let $\theta_\A\colon K_0(\A)\rightarrow K_1(S\A)$ be the Bott map given in \cite[Definition
9.1.1]{Bl} and let $i_\A\colon K_1(S\A)$ $\rightarrow K_1(C(\mathbf S^1,\A))$ be the inclusion.
Then $h=\partial_0\circ\hat\tau^1_*\circ i_\A\circ\theta_\A\in\Hom(K_0(\A),K_0(\B))$ and
$h([1_\A])=\partial_0([\hat v])=e$, where $\hat\tau^1_*$ is the induced map of $\hat\tau$
on $K_1(C(\mathbf S^1,\A))$. This proves (1).

(2) Let $w\in U(\M_2(\B))$ such that $\pi_2(w)=\diag(v,v^*)$, where $\pi_k$ is the induced
homomorphism of $\pi$ on $\M_k(\B)$. Put
$$
w_1=(u_1,u_2)w(u_1,u_2)^T\ \text{and}\
v_0=(\pi(u_1),\pi(u_2))\diag(v_1,v_2)(\pi(u_1),\pi(u_2))^T.
$$
Then $\pi(w_1)=v_0$ and
\begin{align*}
Ad_v\circ\tau_0\oplus Ad_{v^*}\circ\tau_{\A,\B}=&Ad_{\pi(w_1)}\circ(\tau_0\oplus\tau_{\A,\B})\\
Ad_v\circ\tau_{\A,\B}\oplus Ad_{v^*}\circ\tau_{\A,\B}=&
Ad_{\pi(w_1)}\circ(\tau_{\A,\B}\oplus\tau_{\A,\B}).
\end{align*}
Thus, $[Ad_v\circ\tau_0]_u=[Ad_v\circ\tau_{\A,\B}]_u$ in $\Ext_u(\A,\B)$.
\end{proof}

Let $\N$ be the bootstrap category defined in \cite{Bl} or \cite{Lin3}. Then for any $\A\in\N$,
we have following exact sequence (UCT) (cf. \cite{RS}):
$$
0\rightarrow\Ext_\Z(K_*(\A),K_*(\B))\stackrel{\kappa^{-1}}{\longrightarrow}\Ext(\A,\B)
\stackrel{\Gamma_{\A,\B}}{\longrightarrow}
\Hom(K_0(\A),K_1(\B))\oplus\Hom(K_1(\A),K_0(\B))\rightarrow 0
$$
where $\Gamma_{\A,\B}([\tau])=(\partial_1\circ\tau_*^0,\partial_0\circ\tau_*^1)$ and $\tau^i_*$
is the induced homomorphism of $\tau$ on $K_i(\A)$, $i=0,1$, $\kappa$ is a bijective natural map
from $\Ker\Gamma_{\A,\B}$ onto $\Ext_\Z(K_*(\A),K_*(\B))=\Ext_\Z(K_0(\A),K_0(\B))\oplus
\Ext_\Z(K_1(\A),K_1(\B))$.

By means of (UCT), we can obtain our main result in the paper as follows.

\begin{theorem}\label{th2a}
Suppose that $\A$ is a separable nuclear unital $C^*$--algebra and $\B$ is a separable $C^*$--algebra.
If $\A\in\N$, then we have following exact sequence of groups:
\begin{align*}
0\longrightarrow H_1(K_0(\A),K_0(\B))\stackrel{j_\B}{\longrightarrow}&K_0(\B)
\stackrel{\Phi_{\A,\B}}{\longrightarrow}\Ext_u(\A,\B)\longrightarrow\\
\stackrel{i_{\A,\B}}{\longrightarrow}&\Ext(\A,\B)\stackrel{\rho_{\A,\B}}{\longrightarrow}
H_1(K_0(\A),K_1(\B))\longrightarrow 0,
\end{align*}
where, $j_\B$ is an inclusion, $i_{\A,\B}([\tau]_u)=[\tau]$, $\forall\,\tau\in\Hom_1(\A,Q(\B))$
and $\rho_{\A,\B}([\tau])=\partial_1([\tau(1_\A)])$, $\forall\,\tau\in\Hom(\A,Q(\B))$.
\end{theorem}
\begin{proof}Let $[u]\in H_1(K_0(\A),K_1(\B))$. Then there is $h\in\Hom(K_0(\A),K_1(\B))$ such that
$[u]=h([1_\A])\in K_1(\B)$. By (UCT), there is $\tau\in\Hom(\A,Q(\B))$ such that $\Gamma_{\A,\B}
([\tau])=(h,0)$. Therefore,
$$
[u]=h([1_\A])=\partial_1(\tau^0_*([1_\A]))=\partial_1([\tau(1_\A)])=\rho_{\A,\B}([\tau]),
$$
that is, $\rho_{\A,\B}$ is surjective.

Since $[1_{M(\B)}]=0$ in $K_0(M(\B))$ implies that $[1_{Q(\B)}]=0$ in $K_0(Q(\B))$, we have
$\rho_{\A,\B}([\tau])=0$ when $\tau$ is a unital extension and hence $\Ran(i_{\A,\B})\subset
\Ker\rho_{\A,\B}$. Now let $[\tau]\in\Ker\rho_{\A,\B}$ and put $q=\tau(1_\A)$. Then
$\partial_1([q])=0$ and hence by Lemma \ref{L2b}, there is a unital extension $\tau_0$ such
that $\tau\oplus\tau_{\A,\B}\oplus 0\sim_u\tau_0\oplus 0$. Thus, $[\tau]=[\tau_0]$ in
$\Ext(\A,\B)$, i.e., $\Ker\rho_{\A,\B}\subset\Ran(i_{\A,\B})$.

Let $v\in U(Q(\B))$. we can pick $u\in U(\M_2(M(\B)))$ such that $\diag(v,v^*)=\pi_2(u)$.
It follows from $\pi_2(u)\diag(\tau_{\A,\B},0)\pi_2(u^*)=\diag(Ad_v\circ\tau_{\A,\B},0)$
that $\Ran(\Phi_{\A,\B})\subset\Ker i_{\A,\B}$. On the other hand, let $\tau$ be a unital
extension such that $\tau\oplus\tau_1\sim_u\tau_2$ for some trivial extensions $\tau_1$ and
$\tau_2$. If $\tau_1,\tau_2$ are all unital, then $[\tau]_u=0$; if $\tau_1,\tau_2$ are all
non--unital, then by Lemma \ref{L2b}, we can find unital trivial extensions $\tau_1'$ and
$\tau_2'$ such that $\tau_i\oplus\tau_{\A,\B}\oplus 0\sim_u\tau_i'\oplus 0$, $i=1,2$. Consequently,
$\tau\oplus\tau_1'\oplus 0\sim_u\tau_2'\oplus 0$ and hence
$$
\tau\oplus\tau_{\A,\B}\oplus 0\sim_u\tau\oplus\tau_1'\oplus\tau_{\A,\B}\oplus 0\sim_u
\tau_2'\oplus\tau_{\A,\B}\oplus 0\sim_u\tau_{\A,\B}\oplus 0\sim_u\tau_{\A,\B}\oplus
\tau_{\A,\B}\oplus 0,
$$
i.e., there is $u\in U(M(\B))$ such that $\tau\oplus\tau_{\A,\B}\oplus 0=Ad_{\pi(u)}\circ
(\tau_{\A,\B}\oplus\tau_{\A,\B}\oplus 0)$. Let $u_1,u_2$ be isometries in $M(\B)$ such that
$u_1u_1^*+u_2u_2^*=1$. Set $w=(u_1,u_2)^Tu(u_1,u_2)\in U(\M_2(M(\B)))$. Then
$$
\diag(\tau\oplus\tau_{\A,\B}(a),0)=\pi_2(w)\diag(\tau_{\A,\B}\oplus\tau_{\A,\B}(a),0)\pi_2(w)^*,\
\forall\,a\in\A.
$$
It follows that $\pi_2(w)$ has the form $\pi_2(w)=\diag(v_1,v_2)$ and hence
$\tau\oplus\tau_{\A,\B}=Ad_{v_1}\circ(\tau_{\A,\B}\oplus\tau_{\A,\B})$. Let $w_0\in U(M(\B))$
such that $\tau_{\A,\B}\oplus\tau_{\A,\B}=Ad_{\pi(w_0)}\circ\tau_{\A,\B}$. Put $x=\partial_0
([v_1\pi(w_0)])$. Then $\Phi_{\A,\B}(x)=[Ad_{v_1\pi(w_0)}\circ\tau_{\A,\B}]_u=[\tau\oplus
\tau_{\A,\B}]_u=[\tau]_u$.

By Lemma \ref{L2c} (1), $\Ker\Phi_{\A,\B}\subset\Ran(j_\B)$. We now prove $\Ran(j_\B)\subset
\Ker\Phi_{\A,\B}$.

Let $x\in H_1(K_0(\A),K_0(\B))$. Then there is $h\in\Hom(K_0(\A),K_0(\B))$
such that $x=h([1_\A])$. Let $p_\A\colon K_1(C(\mathbf S^1,\A))\rightarrow K_1(S\A)$ be the
projective map. Set $h_0=h\circ\theta_\A^{-1}\circ p_\A$. Then
$h_0\in\Hom(K_1(C(\mathbf S^1,\A)),K_0(\B))$ with
$h_0([z1_\A])=h\circ\theta^{-1}_A([z1_\A])=h([1_\A])=x$. Thus, by (UCT), there is $\tilde{\tau}
\in\Hom(C(\mathbf S^1,\A),Q(\B))$ such that $\Gamma_{C(\mathbf S^1,\A),\B}([\tilde{\tau}])=
(0,h_0)$, i.e., $\partial_0\circ\hat{\tau}_*^1=h_0$ and $\partial_1\circ\hat{\tau}^0_*=0$.
So $[\tilde{\tau}(1_\A)]=0$ in $K_0(Q(\B))$. In this case, $\tilde{\tau}$ can be chosen as
the unital one by Lemma \ref{L2b}. Set $v=\tilde{\tau}(z1_\A)$ and let $\tau$ be the restriction
of $\tilde{\tau}$ on $\A$. Then $x=\partial_0([v])$ and $v\tau(a)=\tau(a)v$, $\forall\,a\in\A$.
Pick $w\in U(\M_2(M(\B)))$ such that $\pi_2(w)=\diag(v,v^*)$. Note that
$$
\pi_2(w)\diag(\tau(a),\tau_{\A,\B}(a))\pi_2(w)^*=\diag(v\tau(a)v^*,v^*\tau_{\A,\B}(a)v),\
\forall\,a\in\A.
$$
We have
$$[\tau]_u=[Ad_v\circ\tau]_u+[Ad_{v^*}\circ\tau_{\A,\B}]_u=[\tau]_u+
[Ad_{v^*}\circ\tau_{\A,\B}]_u
$$
which implies that $[Ad_{v^*}\circ\tau_{\A,\B}]_u=0$ in
$\Ext_u(\A,\B)$ since $\Ext_u(\A,\B)$ is a group. So
$\phi_{\A,\B}(x)=[Ad_v\circ\tau_{\A,\B}]_u=-[Ad_{v^*}\circ\tau_{\A,\B}]_u=0$.
\end{proof}

Let $x\!\in\! K_0(\B)$. Write $[x]\!_{\A,\B}$ to denote the equivalence class of $x$ in
$K_0(\B)/H_1(K_0(\A),\!K_0(\B))$. Since $\Ker\Phi_{\A,\B}=H_1(K_0(\A),K_0(\B))$ by Theorem
\ref{th2a}, we can define the homomorphism
$$
\Phi_{\A,\B}'\colon K_0(\B)/H_1(K_0(\A),K_0(\B))\rightarrow\Ext_u(\A,\B)
$$
by $\Phi_{\A,\B}'([x]_{\A,\B})=\Phi_{\A,\B}(x)$. Thus, we have
\begin{corollary}\label{ca}
Let $\A,\,\B$, $i_{\A,\B}$ and $\rho_{\A,\B}$ be as in Theorem \ref{th2a}. Then we have
following exact sequence of groups:
$$
0\longrightarrow K_0(\B)/H_1(K_0(\A),K_0(\B))\stackrel{\Phi'_{\A,\B}}{\longrightarrow}
\Ext_u(\A,\B)\stackrel{i_{\A,\B}}{\longrightarrow} \Ext(\A,\B)
\stackrel{\rho_{\A,\B}}{\longrightarrow}H_1(K_0(\A),K_1(\B))\longrightarrow 0.
$$
\end{corollary}

\section{The half--exactness of $\Ext_u(\cdot,\B)$ and the (UCT) for $\Ext_u(\A,\B)$}

Let $\J$ be a closed ideal of the separable nuclear unital $C^*$--algebra $\A$ and let
$q\colon\A\rightarrow\A/\J=\C$ be the quotient map. Then $\J$ and $\C$ are nuclear. Define
$q^*\colon\Ext(\C,\B)\rightarrow\Ext(\A,\B)$ and $q_u^*\colon\Ext_u(\C,\B)\rightarrow
\Ext_u(\A,\B)$ respectively, by $q^*([\tau])=[\tau\circ q]$ and $q_u^*([\tau]_u)=[\tau\circ q]_u$.
Let $i^*$ (resp. $i^*_u$) be the induced homomorphism of the inclusion $i\colon\J\rightarrow\A$
on $\Ext(\A,\B)$ (resp. $\Ext_u(\A,\B)$).
\begin{proposition}\label{pa}
Let $\J$, $\A$ and $\C$ be as above. Suppose that $\A$ and $\C$ are all in $\N$. Then
\begin{equation}\label{eqa}
\Ext_u(\C,\B)\stackrel{q_u^*}{\longrightarrow}\Ext_u(\A,\B)\stackrel{i^*_u}{\longrightarrow}
\Ext(\J,\B)
\end{equation}
is exact in the middle $($i.e., $\Ran(q_u^*)=\Ker i^*_u)$. In addition, if there exists a unital
$*$--homomorphism $r\colon\C\rightarrow\A$ such that $q\circ r=\mathrm{id}_\C$, then
\rm{(\ref{eqa})} is split exact.
\end{proposition}
\begin{proof}It is clear that $H_1(K_0(\C),K_i(\B))\subset H_1(K_0(\A),K_i(\B))\subset K_i(\B)$, $i=0,1$.
Let
$$
i_1\colon H_1(K_0(\C),K_1(\B))\rightarrow H_1(K_0(\A),K_1(\B))
$$
be the inclusion and define
$$
i_0\colon K_0(\B)/H_1(K_0(\C),K_0(\B))\rightarrow K_0(\B)/H_1(K_0(\A),K_0(\B)
$$
by $i_0([x]_{\C,\B})=[x]_{\A,\B}$. Let $x\in K_0(\B)$ and $v\in U(Q(\B))$ such that
$\Phi_{\A,\B}(x)=[Ad_v\circ\tau_{\A,\B}]_u$. Since $\tau_{\C,\B}\circ q\in\Hom_1(\A,Q(\B))$ is
trivial extension, we have by Lemma \ref{L2c} (2),
$$
q^*_u\circ\Phi_{\C,\B}'([x]_{\C,\B})=[Ad_v\circ\tau_{\C,\B}\circ q]_u=[Ad_v\circ\tau_{\A,\B}]_u
=\Phi_{\A,\B}(x)=\Phi_{\A,\B}'\circ i_0([x]_{\C,\B}).
$$
It is easy to check that $q^*\circ i_{\C,\B}=i_{\A,\B}\circ q^*_u$, $i^*\circ i_{\A,\B}=i^*_u$
and $i_1\circ\rho_{\C,\B}=\rho_{\A,\B}\circ q^*$. So we get following commutative diagram of
Abelian groups:
\begin{equation}\label{eqb}
\begin{CD}
0@. 0@.\\
@V VV @V VV\\
\dfrac{K_0(\B)}{H_1(K_0(\C),K_0(\B))}@>i_0>>\dfrac{K_0(\B)}{H_1(K_0(\A),K_0(\B))}\\
@V\Phi'_{\C,\B}VV @V\Phi'_{\A,\B}VV\\
\Ext_u(\C,\B)@>q_u^*>>\Ext_u(\A,\B)@>i_u^*>>\Ext(\J,\B)\\
@V i_{\C,\B} VV @V i_{\A,\B}VV @|\\
\Ext(\C,\B)@>q^*>>\Ext(\A,\B)@>i^*>>\Ext(\J,\B)\\
@V\rho_{\C,\B}VV @V\rho_{\A,\B}VV\\
H_1(K_0(\C),K_1(\B))@>i_1>> H_1(K_0(\A),K_1(\B))\\
@V VV @V VV\\
0@. 0@.
\end{CD}
\end{equation}
in which $\Ran(q^*)=\Ker i^*$ by \cite[Theorem 15.11.2]{Bl} and two columns are exact by
Corollary \ref{ca}.

Since $q\circ i=0$, we have $i^*_u\circ q_u^*=0$ and hence $\Ran(q_u^*)\subset\Ker i_u^*$.
Noting that $i_0$ is surjective and $i_1$ is injective, we can use the commutative diagram
(\ref{eqb}) to obtain that $\Ker i_u^*\subset\Ran(q_u^*)$. Thus, (\ref{eqa}) is exact in the
middle.

If the $*$--homomorphism $r\colon\C\rightarrow\A$ satisfies $q\circ r=\rm{id}_\C$, then
$r^*_u\circ q_u^*=\rm{id}$ on $\Ext_u(\C,\B)$ and $i_0$, $i_1$ are all identity maps. So
$q_u^*$ is injective. Since $i^*$ is surjective by \cite[Theorem 15.11.2]{Bl}, there is
$[\tau']\in\Ext(\A,\B)$ such that $i^*([\tau'])=[\tau]$. Pick $[\tau{''}]\in\Ext(\C,\B)$
with $\rho_{\C,\B}([\tau{''}])=i_1^{-1}\circ\rho_{\A,\B}([\tau'])$. Using $\rho_{\A,\B}\circ
q^*=i_1\circ\rho_{\C,\B}$, we get that $\rho_{\A,\B}\circ q^*([\tau''])=\rho_{\A,\B}([\tau'])$.
Thus, there is $[\tau_0]_u\in\Ext_u(\A,\B)$ such that
$i_{\A,\B}([\tau_0]_u)=[\tau']-q^*([\tau''])$. From $i^*\circ i_{\A,\B}=i_u^*$, $\Ran(q^*)=
\Ker i^*$, $i^*([\tau'])=[\tau]$, we get that $[\tau]=i_u^*([\tau_0])$, i.e., $i^*_u$ is
surjective. In this case, (\ref{eqa}) becomes an exact sequence:
\begin{equation}\label{eqc}
0\longrightarrow\Ext_u(\C,\B)\stackrel{q^*_u}{\longrightarrow}\Ext_u(\A,\B)
\stackrel{i^*_u}{\longrightarrow}\Ext(\J,\B)\longrightarrow 0
\end{equation}
with $r_u^*\circ q_u^*=\rm{id}$. Define homomorphism $\rho\colon\Ext(\J,\B)\rightarrow\Ext_u(\A,\B)$
by
$$
\rho([\sigma])=[\sigma']_u-q_u^*\circ r_u^*([\sigma']_u)\quad \text{with}\quad
[\sigma]=i_u^*([\sigma']_u).
$$
It is easy to check that $\rho$ is well--defined and $i^*_u\circ\rho=\rm{id}$ on $\Ext(\J,\B)$.
Thus, (\ref{eqc}) is split exact.
\end{proof}

Let $\Ext_\Z(K_0(\A),[1_\A],K_0(\B))$ be the set of isomorphism classes of all extensions of
$K_0(\A)$ by $K_0(\B)$ with base point of the form
$$
0\longrightarrow K_0(\B)\longrightarrow (G,g_0)\stackrel{\phi}{\longrightarrow}
(K_0(\A),[1_\A])\longrightarrow 0,
$$
where $\phi(g_0)=[1_\A]$. The natural map from $\Ext_\Z(K_0(\A),[1_\A],K_0(\B))$ to
$\Ext_\Z(K_0(\A),K_0(\B))$ has the kernel isomorphic to $K_0(\B)/H_1(K_0(\A),K_0(\B))$ (cf.
\cite[P584]{BD}). Define the homomorphism
$$
\bar\Gamma_{\A,\B}\colon\Ext_u(\A,\B)\rightarrow\Hom_0(K_0(\A),K_1(\B))\oplus
\Hom(K_1(\A),K_0(\B)),
$$
by $\bar\Gamma_{\A,\B}([\tau]_u)=(\partial_1\circ\tau_*^0,\partial_0\circ\tau_*^1)$ for
$\partial_1\circ\tau_*^0([1_\A])=0$ ($[1_{Q(\B)}]=0$ in $K_0(Q(\B))$), where
$\Hom_0(K_0(\A),K_1(\B))=\{h\in\Hom(K_0(\A),K_1(\B)\vert\,h([1_\A])=0\}$.

Let $[\tau]_u\in\Ker\bar\Gamma_{\A,\B}$. We have following isomorphism classes of extensions
\begin{align*}
0\longrightarrow K_0(\B)\longrightarrow (K_0(E),&[1_E])\stackrel{\phi^0_*}{\longrightarrow}
(K_0(\A),[1_\A])\longrightarrow 0\\
0\longrightarrow K_1(\B)\longrightarrow &K_1(E)\stackrel{\phi^1_*}{\longrightarrow}
K_1(\A)\longrightarrow 0,
\end{align*}
where $E=\{(a,b)\in\A\oplus M(\B)\vert\,\pi(b)=\tau(a)\}$, $\phi(a,b)=a$, $\forall\,(a,b)\in E$.
So there exists a natural map $\bar\kappa$ from $\Ker\bar\Gamma_{\A,\B}$ to
$\Ext_\Z(K_0(\A),[1_\A],K_0(\B))\oplus\Ext_\Z(K_1(\A),K_1(\B))\triangleq
\Ext_\Z(K_*(\A),[1_\A],K_*(\B))$. Moveover, we have
\begin{proposition}\label{pb}
Suppose that $\A$ is a separable nuclear unital $C^*$--algebra and $\B$ is a separable stable
$C^*$--algebra. If $\A\in\N$, then $\bar\kappa$ is bijective and the following sequence of groups
is exact:
\begin{align*}
0\longrightarrow\Ext_\Z(K_*(\A),&[1_\A],K_*(\B))\stackrel{\bar\kappa^{-1}}{\longrightarrow}
\Ext_u(\A,\B)\stackrel{\bar\Gamma_{\A,\B}}{\longrightarrow}\\
\longrightarrow&\, \Hom_0(K_0(\A),K_1(\B))\oplus\Hom(K_1(\A),K_0(\B))\longrightarrow 0.
\end{align*}
\end{proposition}
\begin{proof}Let $(h_0,h_1)\in\Hom_0(K_0(\A),K_1(\B))\oplus\Hom(K_1(\A),K_0(\B))$. Then we can
find $[\tau]\in\Ext(\A,\B)$ such that $\Gamma_{\A,\B}([\tau])=(h_0,h_1)$ by (UCT). Since
$h_0([1_A])=\partial_1\circ\tau_*^0([1_\A])=0$, we can pick a unital extension $\tau_0$ with
$[\tau]=[\tau_0]$ in $\Ext(\A,\B)$ by Lemma \ref{L2b}. Thus,
$$
\bar\Gamma_{\A,\B}([\tau_0]_u)=\Gamma_{\A,\B}([\tau_0])=\Gamma_{\A,\B}([\tau])=(h_0,h_1),
$$
that is, $\bar\Gamma_{\A,\B}$ is surjective.

When $[\tau]\in\Ker\Gamma_{\A,\B}$, $[\tau(1_\A)]=0$ in $K_0(Q(\B))$. So there is a unital
extension $\tau_1$ such that $[\tau_1]=[\tau]$ in $\Ext(\A,\B)$ and hence
$i_{\A,\B}([\tau_1]_u)=[\tau]$, $[\tau_1]_u\in\Ker\bar\Gamma_{\A,\B}$, i.e., $i_{\A,\B}\colon
\Ker\bar\Gamma_{\A,\B}\rightarrow\Ker\Gamma_{\A,\B}$ is surjective. Clearly, $\Ran(\Phi'_{\A,\B}
\subset\Ker\bar\Gamma_{\A,\B}$. Thus, we get the exact sequence
$$
0\longrightarrow K_0(\B)/H_1(K_0(\A),K_0(\B))\stackrel{\Phi'_{\A,\B}}{\longrightarrow}
\Ker\bar\Gamma_{\A,\B}\stackrel{i_{\A,\B}}{\longrightarrow}\Ker\Gamma_{\A,\B}\longrightarrow 0
$$
by Corollary \ref{ca}. Therefore, we can deduce from following commutative diagram that
$\bar\kappa$ is bijective.
$$
\begin{CD}
0@. 0@.\\
@V VV @V VV\\
\dfrac{K_0(\B)}{H_1(K_0(\A),K_0(\B))}@=\dfrac{K_0(\B)}{H_1(K_0(\A),K_0(\B))}\\
@V VV @V VV\\
\Ext_\Z(K_*(\A),[1_\A],K_*(\B))@<\bar\kappa<<\Ker\bar\Gamma_{\A,\B}@> >>\Ext_u(\A,\B)\\
@V VV @V i_{\A,\B}VV @V i_{\A,\B}VV\\
\Ext_\Z(K_*(\A),K_*(\B))@<\kappa<<\Ker\Gamma_{\A,\B}@> >>\Ext(\A,\B)\\
@V VV @V VV\\
0@. 0@.
\end{CD}
$$
\end{proof}

\section{The half--exactness of $\Ext_u(\A,\cdot)$}
\setcounter{equation}{0}

Let $\B_1$, $\B_2$ be separable stable $C^*$--algebras and $\phi$ is a $*$--homomorphism of
$\B_1$ to $\B_2$. Let $\phi_*^i\colon K_i(\B_1)\rightarrow K_i(\B_2)$ be the induced
homomorphism of $\phi$ on $K_i(\B_1)$, $i = 0,1$. Then
$$
\phi^i_*((H_1(K_0(\A),K_i(\B_1)))\subset H_1(K_0(\A),K_i(\B_2)),\ i=0,1.
$$
Define the homomorphism
$$
\hat\phi_*\colon K_0(\B_1)/H_1(K_0(\A),K_0(\B_1))\rightarrow K_0(\B_2)/H_1(K_0(\A),K_0(\B_2))
$$
by $\hat\phi_*([x]_{\A,\B_1})=[\phi^0_*(x)]_{\A,\B_2}$, $\forall\,x\in K_0(\B_1)$.

Let $\B_1, \B_2$ and $\phi$ be as above. If $\overline{\phi(\B_1)\B_2} =\B_2$ (especially,
$\phi$ is surjective), then $\phi$ has a unique extension $\bar\phi$ from $M(\B_1)$ to $M(\B_2)$
such that $\bar\phi$ is strictly continuous and $\bar\phi(1_{M(\B_1)}) = 1_{M(\B_2)}$ by
\cite[Corollary 1.1.15]{JT}. Let $\pi_{\B_i}$ be the quotient map from $M(\B_i)$ onto $Q(\B_i)$,
$i=1,2$ and let $\hat\phi\colon Q(\B_1)\rightarrow Q(\B_2)$ be the unital $*$--homomorphism
induced by $\bar\phi$ such that $\hat\phi\circ\pi_{\B_1}=\pi_{\B_2}\circ\bar\phi$. In this case,
we can define homomorphisms $\phi_*\colon\Ext(\A,\B_1)\rightarrow\Ext(\A,\B_2)$ and
$\phi_*^u\colon\Ext_u(\A,\B_1)\rightarrow\Ext_u(\A,\B_2)$ respectively by
$$
\phi_*([\tau])=[\hat\phi\circ\tau],\ \forall\,[\tau]\in\Ext(\A,\B_1),\quad
\phi_*^u([\tau]_u)=[\hat\phi\circ\tau]_u,\ \forall\,[\tau]_u\in\Ext_u(\A,\B_1).
$$

In general, according to \cite[Lemma 1.3.19, Corollary 1.1.15]{JT}, $\phi$ is homotopic to
a quasi--unital $*$--homomorphism $\phi_0\colon\B_1\rightarrow\B_2$ (i.e.,
$\overline{\phi_0(\B_1)\B_2}=p\B_2$ for some projection $p\in M(\B_2)$) and $\phi_0$ has
a unique extension $\bar\phi_0\colon M(\B_1)\rightarrow M(\B_2)$ with $p=\bar\phi_0(1_{M(\B_1)})$
such that $\bar\phi_0$ is strictly continuous. In this case, we define
$\phi_*([\tau])=[\hat\phi_0\circ\tau],\ \forall\,[\tau]\in\Ext(\A,\B_1)$ and
$\phi_*^u([\tau]_u)=[\hat\phi_0\circ\tau']_u,\ \forall\,\tau\in\Hom_1(\A,Q(\B_1))$, where
$$
\tau'(a)=(\pi_{\B_2}(u_1),\pi_{\B_2}(u_2))\big(\diag(\hat\phi_0\circ\tau(a),0)+
\pi_{\B_2}(W)\diag(\tau_{\A,\B_2}(a),0)\pi_{\B_2}(W^*)\big)
\begin{pmatrix}\pi_{\B_2}(u_1^*)\\ \ \\ \pi_{\B_2}(u_2^*))\end{pmatrix}\!,
$$
$\forall\,a\in\A$ and $W\in\M_2(M(\B_2))$ such that
$W^*W=\diag(1_{M(\B_2)},0)$, $WW^*=\diag(1_{M(\B_2)}-p,1_{M(\B_2)})$, and $u_1,u_2$ are
isometries in $M(\B_2)$ with $u_1u_1^*+u_2u_2^*=1_{M(\B_2)}$. The $\phi_*$
above is well--defined (cf. \cite[Remark 2.9]{Sh}). In order to show
$\phi_*^u$ is well--defined, we need following lemma.
\begin{lemma}\label{L4a}
Let $\sigma_t\colon\A\rightarrow M(\B)$ be unital completely positive maps for all $t$ in $[0,1]$
such that $t\mapsto\sigma_t(a)$ is strictly continuous for every $a$ in $\A$ and $t\mapsto
\sigma_t(ab)-\sigma_t(a)\sigma_t(b)$ is norm--continuous from $[0,1]$ to $\B$ for all $a,b\in\A$.
Put $\tau_t=\pi\circ\sigma_t\in\Hom_1(\A,Q(\B))$, $\forall\,t\in [0,1]$. Suppose that $\A\in\N$.
Then $[\tau_0]_u=[\tau_1]_u$ in $\Ext_u(\A,\B)$.
\end{lemma}
\begin{proof}
Set $I=[0,1]$ and $I\B=\{f\colon I\rightarrow\B\ \text{continuous}\}$. Let $\Lambda_i\colon I\B
\rightarrow\B$ be $\Lambda_i(f)=f(i)$, $i=0,1$, $\forall\,f\in I\B$.
 Since $M(I\B)=\{f\colon I\rightarrow M(\B)\ \text{strictly continuous}\}$ and $\Lambda_i$ is
 surjective, we have $\bar\Lambda_i(f)=f(i)$, $i=0,1$, $\forall\,f\in M(I\B)$. Note that
 $\Lambda_{j*}^i\colon K_i(\I\B)\rightarrow K_i(\B)$ is isomorphic, $i,j=0,1$ and $\Lambda_{j*}
 \colon\Ext(\A,I\B)\rightarrow\Ext(\A,\B)$ is also isomorphic by \cite[Theorem 19.5.7]{Bl},
 $j=0,1$ as $C_0((0,1],\B)$, $C_0([0,1),\B)$ are contractible and
 $$
 0\longrightarrow C_0((0,1],\B)\longrightarrow I\B\longrightarrow\B\longrightarrow 0,\quad
0\longrightarrow C_0([0,1),\B)\longrightarrow I\B\longrightarrow\B\longrightarrow 0
$$
are split exact.

Consider following diagram of exact sequences obtained by Corollary \ref{ca}.
$$
\begin{CD}
0@.\longrightarrow\dfrac{K_0(I\B)}{H_1(K_0(\A),K_0(I\B))}@.\stackrel{\Phi'_{\A,I\B}}{\longrightarrow}
\Ext_u(\A,I\B)@.\stackrel{i_{\A,I\B}}{\longrightarrow}
\Ext(\A,I\B)@.\stackrel{\rho_{\A,I\B}}{\longrightarrow}H_1(K_0(\A),K_1(I\B))@.\longrightarrow 0\\
@. @V\hat\Lambda_{j*}VV @V\Lambda^u_{j*}VV @V\Lambda_{j*}VV @V\Lambda^1_{j*}VV\\
0@.\longrightarrow\dfrac{K_0(\B)}{H_1(K_0(\A),K_0(\B))}@.\stackrel{\Phi'_{\A,\B}}{\longrightarrow}
\Ext_u(\A,\B)@.\stackrel{i_{\A,\B}}{\longrightarrow}
\Ext(\A,\B)@.\stackrel{\rho_{\A,\B}}{\longrightarrow}H_1(K_0(\A),K_1(\B))@.\longrightarrow 0\\
\end{CD}.
$$
According to the definitions of homomorphisms in the diagram, it is easy to verify that this
diagram is commutative. Since $\hat\Lambda_{j*}$, $\Lambda_{j*}$ and $\Lambda^1_{j*}$ are all
isomorphic, we have $\Lambda^u_{j*}$ is isomorphic by 5--Lemma, $j=0,1$.

Let $R\colon\B\rightarrow I\B$ be given by $R(x)(t)=x$, $\forall\,x\in\B$ and $t\in I$. Then
$\Lambda_j\circ R=\rm{id}_\B$ and $\overline{R(\B)(I\B)}=I\B$, $j=0,1$. So $R_*\colon\Ext_u(\A,\B)
\rightarrow\Ext_u(\A,I\B)$ is well--defined and $\Lambda^u_{j*}\circ R^u_*=\rm{id}$ on
$\Ext_u(\A,\B)$, $j=0,1$. Therefore, $\Lambda^u_{0*}=\Lambda^u_{1*}$ by above argument.

Now the map $\tilde\sigma(a)(t)=\sigma_t(a)$, $\forall\,a\in\A$ and $t\in I$ defines a unital
completely map from $\A$ to $M(I\B)$ with $\tilde\sigma(ab)-\tilde\sigma(a)\tilde\sigma(b)
\in I\B$, $\forall\,a,b\in\A$. Put $\tilde\tau=\pi_{I\B}\circ\tilde\sigma$. Then $\tilde\tau
\in\Hom_1(\A,Q(I\B))$ and $\Lambda^u_{j*}([\tilde\tau]_u)=[\tau_j]_u$, $j=0,1$, so that
$[\tau_0]_u=[\tau_1]_u$.
\end{proof}

\begin{proposition}\label{pc}
Let $\A\in\N$ be a unital separable nuclear $C^*$--algebra and $\B_1,\B_2$ be separable
stable $C^*$--algebras. Let $\phi\colon\B_1\rightarrow\B_2$ be a $*$--homomorphism. We have
\begin{enumerate}
\item[$(1)$] $\Ext_u(\cdot,\B)$ is homotopic invariant for first variable in the class of
separable nuclear algebras belonging to $\N;$
\item[$(2)$] $\phi^u_*\colon\Ext_u(\A,\B_1)\rightarrow\Ext_u(\A,\B_2)$ given above is
well--defined\,$;$
\item[$(3)$] $\Ext_u(\A,\cdot)$ is homotopic invariant for second variable in the class of
separable, stable $C^*$--algebras\,$;$
\item[$(4)$] $\phi_*\circ i_{\A,\B_1}=i_{\A,\B_2}\circ\phi^u_*$, $\phi^u_*\circ\Phi'_{\A,\B_1}
=\Phi'_{\A,\B_2}\circ\hat\phi_*$.
\end{enumerate}
\end{proposition}
\begin{proof}(1) Let $\A_1$, $\A_2$ be unital separable nuclear $C^*$--algebras which is in $\N$.
Let $\alpha_1$, $\alpha_2$ be unital $*$--homomorphisms from $\A_1$ to $\A_2$. Suppose that
there is a path of unital $*$--homomorphism $\rho_t$ from $\A_1$ to $\A_2$ for all $t\in[0,1]$
such that $t\mapsto\rho_t(a)$ is continuous from $[0,1]$ to $\A_2$ for ever $a\in\A_1$ and
$\rho_0=\alpha_1$, $\rho_1=\alpha_2$.

Let $\tau\in\Hom_1(\A,Q(\B))$. Then there is a unital completely positive map $\sigma\colon
\A_2\rightarrow M(\B)$ such that $\tau=\pi\circ\tau$. Set $\sigma_t=\sigma\circ\rho_t$ and
$\tau_t=\pi\circ\sigma_t$, $\forall\,t\in [0,1]$. Obviously, $\{\sigma_t\vert\,t\in[0,1]\}$
satisfies the conditions given in Lemma \ref{L4a}. Thus, $\alpha_1^*([\tau]_u)=
[\tau\circ\rho_0]_u=[\tau\circ\rho_1]_u=\alpha_2^*([\tau]_u)$.

(2) If $\phi$ is homotopic to another quasi--unital $*$--homomorphism $\phi'$, then by
\cite[Lemma 3.1.15]{JT}, there is a path $\lambda_t$ of $*$--homomorphism from $M(\B_1)$
to $M(\B_2)$, $\forall\,t\in [0,1]$ such that
\begin{enumerate}
\item[(a)] $t\mapsto\lambda_t(x)$ is strictly continuous for every $x\in M(\B_1)$;
\item[(b)] $t\mapsto\lambda_t(b)$ is norm--continuous from $[0,1]$ to $\B_2$ for any $b\in \B_1$;
\item[(c)] $\lambda_0=\bar\phi_0$, $\lambda_1=\bar\phi'$.
\end{enumerate}
Put $\hat p(t)=\lambda_t(1_{M(\B_1)})$. Then $\hat p$ is a projection in $M(I\B_2)$ and hence
there ia partial isometry $\hat w$ in $\M_2(M(I\B_2))$ such that
$$
\hat w^*\hat w=\diag(1_{M(I\B_2)},0)\ \text{and}\ \hat w\hat w^*=\diag(1_{M(I\B_2)}-\hat p,
1_{M(I\B_2)}).
$$
Simple computation shows that $W$ (resp. $\hat w$) has the form $W=\begin{pmatrix}w_1&0\\
w_2&0\end{pmatrix}$ (resp.\! $\hat w=\begin{pmatrix}\hat w_1&0\\ \hat w_2&0\end{pmatrix}$\!) with
\begin{alignat*}{4}
w_1w_1^*&=1_{M(\B_2)}-p,\ w_1w_2^*&=0,\ w_2w_2^*&=1_{M(\B_2)},\ w_1^*w_1+w_2^*w_2&=1_{M(\B_2)};\\
\hat w_1\hat w_1^*&=1_{M(\B_2)}-\hat p,\ \hat w_1\hat w_2^*&=0,\ \hat w_2\hat w_2^*&=1_{M(\B_2)},\
\hat w_1^*\hat w_1+\hat w_2^*\hat w_2&=1_{M(\B_2)}.
\end{alignat*}
Set $v=w_1^*\hat w_1(0)+w_2^*\hat w_2(0)$. Then $v$ is unitary in $M(\B_2)$ and $w_1v=\hat w_1(0)$,
$w_2v=\hat w_2(0)$. Let $v(t)$ be a strictly continuous path in $U(M(\B_2))$ ($t\in [0,1]$)
such that $v(0)=v$, $v(1)=1_{M(\B_2)}$. Set
$$
\hat W_i(t)=\begin{cases}\hat w_i(0)v^*(2t)\quad &0\le t\le 1/2\\
                         \hat w_i(2t-1)\quad &1/2\le t\le 1
                         \end{cases},\quad i=1,2,\quad
\hat W=\begin{pmatrix}\hat W_1&0\\ \hat W_2&0\end{pmatrix}.
$$
Then $\hat W\in M(I\B_2)$ with $\hat W(0)=W$. Set
$$
\hat q(t)=\begin{cases} p\quad &0\le t\le 1/2\\ \hat p(2t-1)\quad &1/2\le t\le 1\end{cases},\quad
\hat\lambda_t=\begin{cases}\bar\phi_0\quad &0\le t\le 1/2\\ \lambda_{2t-1}\quad &1/2\le t\le 1
\end{cases}.
$$
Clearly, $\hat q(t)=\hat\lambda_t(1_{M(\B_2)})$ and $\hat\lambda_t$ satisfies Condition (a),
(b) and (c), and
$$
\hat W^*\hat W=\diag(1_{M(\B_2)},0),\quad \hat W\hat W^*=\diag(1_{M(\B_2)}-\hat q,1_{M(\B_2)}).
$$
Let $q=\bar\phi'(1_{M(\B_1)})$. Pick $W'\in\M_2(M(\B_2))$ such that $W'W'^*=\diag(1_{M(\B_2)},0)$,
$W'^*W'=\diag(1_{M(\B_2)}-q,1_{M(\B_2)})$. If $W'\not=\hat W(1)$, using above method, we can
choose $\hat\lambda_t$ and $\hat W$ such that $\hat W(1)=W'$.

Let $\tau\in\Hom_1(\A,Q(\B_1))$ with $\tau=\pi_{\B_1}\circ\sigma$ for some unital completely
positive map $\sigma\colon\A\rightarrow M(\B_2)$ and let $\tau_{\A,\B_2}=\pi_{\B_2}\circ\psi$
for some unital $*$--homomorphism $\psi\colon\A\rightarrow M(\B_2)$. Set
$$
\sigma_t(a)=(u_1,u_2)(\diag(\hat\lambda_t\circ\sigma(a),0)+\hat W(t)\diag(\psi(a),0)\hat W^*(t))
(u_1,u_2)^T,\ \forall\,a\in\A.
$$
Note that $\sigma(ab)-\sigma(a)\sigma(b)\in\B_1$, $\forall\,a, b\in\A$ and $\hat\lambda_t$
satisfies Condition (a), (b) and (c). So we have $[\pi_{\B_2}\circ\sigma_0]_u=
[\pi_{\B_2}\circ\sigma_1]_u$ by Lemma \ref{L4a}, that is, $\phi^u_*$ is well--defined.

(3) By (2), we can assume that $\phi_1,\,\phi_2\colon\B_1\rightarrow\B_2$ are quasi--unital
$*$--homomorphisms which are homotopic. Using the same methods as in the proof of (2), we
have $\phi^u_{1*}=\phi^u_{2*}$.

(4) Put $r=\begin{pmatrix}\pi_{\B_2}(p)&0\\ 0&0\end{pmatrix}$ and
$U=\begin{pmatrix}r& 1_2-r\\ 1_2-r&r\end{pmatrix}\in U_0(\M_4(Q(\B_2)))$, where  $1_2$ is the unit
of $\M_2(Q(\B_2))$. Then for any $\tau\in\Hom_1(\A,Q(\B_1))$ and any $a\in\A$,
\begin{align*}
\diag\Big(\begin{pmatrix}\hat\phi_0\circ\tau(a)\\ \ &0\end{pmatrix}+&\pi_{\B_2}^{(2)}(W)\begin{pmatrix}
\tau_{\A,\B_2}(a)\\ &0\end{pmatrix}\pi_{\B_2}^{(2)}(W^*),0\Big)\\
=&U\begin{pmatrix}\hat\phi_0\circ\tau(a)\\ &\pi_{\B_2}^{(2)}(W)\begin{pmatrix}\tau{\A,\B_2}(a)\\
&0\end{pmatrix}\pi^{(2)}_{\B_2}(W^*)\end{pmatrix}U^*,
\end{align*}
where $\pi^{(n)}_{\B_2}$ represents the induced homomorphism of $\pi_{\B_2}$ on
$\M_n(M(\B_2))$. Put
$$
\bar W=\begin{pmatrix}W&1_2-WW^*\\ 1_2-W^*W&W^*\end{pmatrix}\in U(\M_4(M(\B_2))).
$$
Since $1_2-W^*W=\diag(0,1)$, it follows that for every $a\in\A$,
$$
\diag\Big(\pi^{(2)}_{\B_2}(W)\begin{pmatrix}\tau_{\A,\B_2}(a)\\ &0\end{pmatrix}
\pi^{(2)}_{\B_2}(W^*),0\Big)=\pi^{(4)}_{\B_2}(\bar W)\diag\Big(\begin{pmatrix}\tau_{\A,\B_2}(a)\\
&0\end{pmatrix}\!,0\Big)\pi^{(4)}_{\B_2}(\bar W^*).
$$
Therefore, $i_{\A,\B_2}\circ\phi^u_*([\tau]_u)=[\hat\phi_0\circ\tau]=\phi_*\circ i_{\A,\B_1}
([\tau]_u)$ in $\Ext_u(\A,\B_2)$.

Let $x\in K_0(\B_1)$. Then there is $v\in U(Q(\B_1))$ such that $\partial_0([v])=x$ and
$\Phi'_{\A,\B_1}([x]_{\A,\B_1})$ $=[Ad_v\circ\tau_{\A,\B_1}]_u$. Let $\tau_{\A,\B_1}=\pi_{\B_1}
\circ\psi_0$ for some unital $*$--homomorphism $\psi_0\colon\A\rightarrow M(\B_1)$ and set
\begin{align*}
\tilde\psi(a)&=(u_1,u_2)\Big(\begin{pmatrix}\bar\phi_0\circ\psi_0(a)\\ &0\end{pmatrix}+
W\begin{pmatrix}\psi_0(a)\\ &0\end{pmatrix}W^*\Big)(u_1,u_2)^T,\ \forall\,a\in\A,\\
\tilde v&=(\pi_{\B_2}(u_1),\pi_{\B_2}(u_2))\diag(\hat\phi_0(v)\!+\!1_{Q(\B_2)}\!-\!
\pi_{\B_2}(p),1_{Q(\B_2)})(\pi_{\B_2}(u_1),\pi_{\B_2}(u_2))^T\!\in\! U(Q(\B_2)).
\end{align*}
Then $\tilde\psi\colon\A\rightarrow M(\B_2)$ is a unital $*$--homomorphism and
$$
\phi^u_*\circ\Phi'_{\A,\B_1}([x]_{\A,\B_1})=[Ad_{\tilde v}\circ(\pi_{B_2}\circ\tilde\psi)]_u
=[Ad_{\tilde v}\circ\tau_{\A,\B_2}]_u,
$$
by Lemma \ref{L2c}. Noting that $[\tilde v]=[\hat\phi_0(v)+1_{Q(\B_2)}-\pi_{\B_2}(p)]$ in $K_1(Q(\B_2))$
by Lemma \ref{L2a} and
$$
\partial_0([\hat\phi_0(v)+1_{Q(\B_2)}-\pi_{\B_2}(p)])=[\phi^0_{0*}(x)]=[\phi_*^0(x)],
$$
We have
$$
\Phi'_{\A,\B_2}\circ\hat\phi_*([x]_{\A,\B_1})=\Phi'_{\A,\B_2}([\phi^0_*(x)]_{\A,\B_2})
=[Ad_{\tilde v}\circ\tau_{\A,\B_2}]_u=\phi^u_*\circ\Phi'_{\A,\B_1}([x]_{\A,\B_1}).
$$
\end{proof}
\begin{remark}
In the definition of $\phi^u_*$, we can replace $\tau_{\A,\B_2}$ by any unital extension
$\mu\colon\A\rightarrow Q(\B_2)$. Let $\phi_0$, $W$ and $u_1,u_2$ be as above and let
$\mu_0\colon\A\rightarrow Q(\B_2)$ be a unital trivial extension. Put
\begin{align*}
\tau_0'(a)&=\diag(\hat\phi_0\circ\tau(a),0)+\pi_{\B_2}(W)\diag(\tau_{\A,\B_2}(a),0)\pi_{\B_2}(W^*),\\
\tau''_0(a)&=\diag(\hat\phi_0\circ\mu_0(a),0)+\pi_{\B_2}(W)\diag(\mu(a),0)\pi_{\B_2}(W^*),\\
\tau'_1(a)&=\diag(\hat\phi_0\circ\tau(a),0)+\pi_{\B_2}(W)\diag(\mu(a),0)\pi_{\B_2}(W^*),\\
\tau''_1(a)&=\diag(\hat\phi_0\circ\mu_0(a),0)+\pi_{\B_2}(W)\diag(\tau_{\A,\B_2}(a),0)\pi_{\B_2}(W^*),\\
\tau'(a)&=(\pi_{\B_2}(u_1),\pi_{\B_2}(u_2))\tau'_0(a)(\pi_{\B_2}(u_1),\pi_{\B_2}(u_2))^T,\\
\tau''(a)&=(\pi_{\B_2}(u_1),\pi_{\B_2}(u_2))\tau'_1(a)(\pi_{\B_2}(u_1),\pi_{\B_2}(u_2))^T,
\end{align*}
$\forall\,a\in\A$. Since $U\diag(\tau'_0(a),\tau''_0(a))U^*=\diag(\tau'_1(a),\tau''_1(a))$,
$\forall\,a\in\A$, where $U$ is given in the proof of Proposition \ref{pc} (4) and
$$
(\pi_{\B_2}(u_1),\pi_{\B_2}(u_2))\tau''_0(\cdot)\begin{pmatrix}\pi_{\B_2}(u_1^*)\\ \ \\
\pi_{\B_2}(u_2^*)\end{pmatrix},\
\pi_{\B_2}(u_1),\pi_{\B_2}(u_2))\tau''_1(\cdot)\begin{pmatrix}\pi_{\B_2}(u_1^*)\\ \ \\
\pi_{\B_2}(u_2^*)\end{pmatrix}
$$
are all unital trivial extensions, we obtain that $[\tau']_u=[\tau'']_u$ in $\Ext_u(\A,\B_2)$.
\end{remark}

Let $\I$ be a closed ideal of $\B$ and let $\Lambda\colon\B\rightarrow \B/\I=\D$ be the
canonical homomorphism. Then $\I$ and $\D$ are all stable by \cite[Corollary 2.3 (ii)]{Ro}.

\begin{theorem}\label{th4a}
Let the short exact sequence $0\longrightarrow\I\stackrel{j}{\longrightarrow}\B
\stackrel{\Lambda}{\longrightarrow}\D\longrightarrow 0$ be given as above. Suppose that
$\A\in\N$ is a unital separable unclear $C^*$--algebra and
\begin{enumerate}
\item[\rm{(i)}] there is a completely positive map $\Psi\colon\D\rightarrow\B$ such that
$\Lambda\circ\Psi=\rm{id}_\D;$
\item[\rm{(ii)}] $\Ran(\hat j_*)=\Ker\hat\Lambda_*;$
\item[\rm{(iii)}] $j_*^1\colon H_1(K_0(\A),K_1(\I))\rightarrow H_1(K_0(\A),K_1(\B))$ is injective.
\end{enumerate}
Then
\begin{equation}\label{eqx}
\Ext_u(\A,\I)\stackrel{j_*^u}{\longrightarrow}\Ext_u(\A,\B)\stackrel{\Lambda^u_*}{\longrightarrow}
Ext_u(\A,\D)
\end{equation}
is exact in the middle. Especially, if the $\psi$ in Condition \rm{(i)} is a $*$--homomorphism,
then $j^u_*$ is injective, $\Lambda^u_*$ is surjective in \rm{(\ref{eqx})} and \rm{(\ref{eqx})}
is split exact.
\end{theorem}
\begin{proof} Let $j_0\colon\I\rightarrow\B$ be a quasi--unital $*$--homomorphism with
$\bar j_0(1_{M(\I)})=p\in M(\B)$ such that $j_0$ is homotopic to $j$. Then $\Lambda\circ j_0
\colon\I\rightarrow\D$ is homotopic to $\Lambda\circ j=0$. So $(\Lambda\circ j_0)^u_*=0$ by
Proposition \ref{pc}. Let $\tau\in\Hom_1(\A,Q(\I))$. Then $\Lambda^u_*\circ j_*^u([\tau]_u)=
[\hat\Lambda\circ\tau']_u$, where
$$
\tau'(a)=(\pi(u_1),\pi(u_2))\Big(\begin{pmatrix}\hat j_0\circ\tau(a)\\ &0\end{pmatrix}+
\pi^{(2)}(W)\begin{pmatrix}\tau_{\A,\B}(a)\\ &0\end{pmatrix}\pi^{(2)}(W^*)\Big)(\pi(u_1),\pi(u_2))^T,
$$
$\forall\,a\in\A$and $W^*W=\diag(1,0)$, $WW^*=\diag(1-p,1)$. Set $v_i=\pi(u_i)$, $i=1,2$,
$p_0=\bar\Lambda(p)$ and $W_0=\bar\Lambda(W)$. Since $v_1,v_2$ are isometries with
$v_1v_1^*+v_2v_2^*=1$ and $\hat\Lambda\circ\tau_{\A,\B}\colon\A\rightarrow Q(\D)$ is a trivial
unital extension and $W_0^*W_0=\diag(1,0)$, $W_0W^*_0=\diag(1-p_0,1)$ in $\M_2(M(\D))$.
Noting that $\overline{\Lambda\circ j_0}=\bar\Lambda\circ\bar j_0$, we have
$\widehat{\Lambda\circ j_0}=\hat\Lambda\circ\hat j_0$. Therefore,
$\Lambda^u_*\circ j_*^u([\tau]_u)=(\Lambda\circ j_0)^u_*([\tau]_u)=0$, i.e., $\Ran(j^u_*)\subset
\Ker\Lambda^u_*$.

In order to prove $\Ker\Lambda^u_*\subset\Ran(j_*^u)$, we consider following diagram:
\begin{equation}\label{eqy}
\begin{CD}
0@. 0@. 0@.\\
@V VV @V VV @V VV\\
\dfrac{K_0(\I)}{H_1(K_0(\A),K_0(\I)}@>\hat j_*>>\dfrac{K_0(\B)}{H_1(K_0(\A),K_0(\B)}
@>\hat\Lambda_*>>\dfrac{K_0(\D)}{H_1(K_0(\A),K_0(\D)}\\
@V\Phi'_{\A,\I}VV @V\Phi'_{\A,\B}VV @V\Phi'_{\A,\D}VV\\
\Ext_u(\A,\I)@>j^u_*>>\Ext_u(\A,\B)@>\Lambda^u_*>>\Ext_u(\A,\D)\\
@V i_{\A,\I} VV @V i_{\A,\B} VV @V i_{\A,\D} VV\\
\Ext(\A,\I)@>j_*>>\Ext(\A,\B)@>\Lambda_*>>\Ext(\A,\D)\\
@V\rho_{\A,\I}VV @V\rho_{\A,\B}VV @V\rho_{\A,\D}VV\\
H_1(K_0(\A),K_1(\I))@>j_*^1>>H_1(K_0(\A),K_0(\B))@>\Lambda_*^1>>H_1(K_0(\A),K_0(\D))\\
@V VV @V VV @V VV\\
0@. 0@. 0@.
\end{CD}.
\end{equation}
In (\ref{eqy}), three columns are exact by Corollary \ref{ca}, the first row is exact
by Condition (ii) and the third row is exact by \cite[Theorem 19.5.7]{Bl}.
It is easy to check that following two diagrams
$$
\begin{CD}
K_0(Q(\I))@>\hat j_{0*}^0>> K_0(Q(\B))\\
@V\partial_1 VV @V\partial_1 VV\\
K_1(\I)@>j_{0*}^1>> K_1(\B)
\end{CD},\qquad
\begin{CD}
K_0(Q(\B))@>\hat\Lambda_*^0>> K_0(Q(\D))\\
@V\partial_1 VV @V\partial_1 VV\\
K_1(\B)@>\Lambda_*^1>> K_1(\D)
\end{CD}
$$
are commutative. Thus, $\rho_{\A,\B}\circ j_*=j_*^1\circ\rho_{\A,\I}$,
$\Lambda_*^1\circ\rho_{\A,\B}=\rho_{\A,\D}\circ\Lambda_*$. Finally,
the diagram (\ref{eqy}) is commutative by Proposition \ref{pc} (4). We can deduce the
assertion from (\ref{eqy}).
\end{proof}

\begin{corollary}\label{cb}
Let $\I$, $\B$, $\D$ and $j,\ \Lambda$ be as in Theorem \ref{th4a}. Assume that $\D$ is
contractible and there is a completely positive map $\Psi\colon\D\rightarrow\B$ such that
$\Lambda\circ\Psi=\rm{id}_\D$. Then $j^u_*\colon\Ext_u(\A,\I)\rightarrow\Ext_u(\A,\B)$ is
isomorphic.
\end{corollary}
\begin{proof}The assumptions indicate that $\hat j_*$, $j_*$ and $j_*^1$ are all isomorphic.
By using Theorem \ref{th4a}, Proposition \ref{pc} (3) to the commutative diagram (\ref{eqy}),
we can obtain the assertion.
\end{proof}
\begin{corollary}\label{cc}
Let $\A\in\N$ be a unital separable nuclear $C^*$--algebra and $\B$ be a separable nuclear
stable $C^*$--algebra. Then $\Ext_u(\A,S^2\B)\cong\Ext_u(\A,\B)$, where $S^2\B=C_0(\R^2)\otimes
\B$.
\end{corollary}
\begin{proof}
Let $\mathrm S$ be the unilateral shift on $l^2$ and let $T_0=C^*(\mathrm S-I_{l^2})$ be
the $C^*$--subalgebra in $B(l^2)$ generated by $\mathrm S-I_{l^2}$. Then we have a short
exact sequence:
\begin{equation}\label{eqz}
0\longrightarrow\B\stackrel{j}{\longrightarrow}T_0\otimes\B\stackrel{q}{\longrightarrow}
S\B\longrightarrow 0
\end{equation}
and also have $K_i(T_0\otimes\B)\cong 0$, $i=0,1$ (cf. \cite{Cu}). Set
$$C_q=\{(x,f)\in(T_0\otimes\B)\oplus C_0([0,1),S\B)\vert\,q(x)=f(0)\}.
$$
Then we have following exact sequences of $C^*$--algebras
\begin{align}
\label{eqxx} 0\longrightarrow\B\stackrel{e}{\longrightarrow}&C_q\longrightarrow C_0([0,1),SB)
\longrightarrow 0\\
\label{eqyy} 0\longrightarrow S^2\B\stackrel{i}{\longrightarrow}&C_q\longrightarrow T_0\otimes\B
\longrightarrow 0.
\end{align}
Since $C_0([0,1),S\B)$ is contractible, applying Corollary \ref{cb} to (\ref{eqxx}), we get
 that $e^u_*\colon
\Ext_u(\A,\B)$ $\rightarrow\Ext_u(\A,C_q\otimes\K)$ is isomorphic. We also obtain that
$e_*\colon\Ext(\A,\B)\rightarrow\Ext(\A,C_q\otimes\K)$ and $e^i_*\colon K_i(\B)\rightarrow K_i(C_q)$
are all isomorphic, $i=0,1$. Set
$$
\partial_u=(e^u_*)^{-1}\circ i^u_*,\ \hat\partial=(\hat e_*)^{-1}\circ\hat i_*,\
\partial=(e_*)^{-1}\circ i_*,\ \partial^1=(e^1_*)^{-1}\circ i^1_*.
$$
Then we have following commutative diagram of exact sequences by Proposition \ref{pc}:
$$
\begin{CD}
0@.\rightarrow\dfrac{K_0(\B)}{H_1(K_0(\A),K_0(\B))}@.\stackrel{\Phi'_{\A,\B}}{\longrightarrow}
\Ext_u(\A,\B)@.\stackrel{i_{\A,\B}}{\longrightarrow}
\Ext(\A,\B)@.\stackrel{\rho_{\A,\B}}{\longrightarrow}H_1(K_0(\A),K_1(\B))@.\rightarrow 0\\
@. @A \hat\partial AA @A\partial_u AA @A\partial AA @A\partial^1 AA\\
0@.\rightarrow\dfrac{K_0(S^2\!\B)}{H_1(K_0(\A),K_0(S^2\!\B))}@.\stackrel{\Phi'_{\A,\!S^2\!\B}}{\longrightarrow}
\Ext_u(\A,\!S^2\!\B)@.\stackrel{i_{\A,\!S^2\!\B}}{\longrightarrow}
\Ext(\A,\!S^2\!\B)@.\stackrel{\rho_{\A,\!S^2\!\B}}{\longrightarrow}H_1(K_0(\A),K_1(S^2\!\B))@.\rightarrow 0\\
\end{CD}.
$$
Since $K_i(T_0\otimes\B)\cong 0$, $i=0,1$, it follows from (\ref{eqyy}) that $\hat i_*$, $i_*$ and
$i^1_*$ are all isomorphic. Thus, $\hat\partial$, $\partial$ and $\partial^1$ are isomorphic
in above commutative diagram and consequently, $\partial_u$ is isomorphic.
\end{proof}
\vspace{2mm}

\noindent{\bf Acknowledgement}\quad The author is grateful to Professor Huaxin Lin for his
helpful suggestions while preparing this paper.

\end{document}